\UseRawInputEncoding
\documentclass{amsart}
\usepackage{amsthm,hyperref}
\usepackage{todonotes}
\usepackage{soul}

\newtheorem{theorem}{Theorem}[section]
\newtheorem{corollary}[theorem]{Corollary}
\newtheorem{lemma}[theorem]{Lemma}
\newtheorem{proposition}[theorem]{Proposition}
\theoremstyle{definition}
\newtheorem{remark}[theorem]{Remark}

\newcommand{\RR}{\mathbb{R}}

\newcommand{\CP}{\mathbb{CP}}
\newcommand{\HP}{\mathbb{HP}}

\newcommand{\V}{\mathcal{V}}

\newcommand{\g}{\mathfrak{g}}
\newcommand{\h}{\mathfrak{h}}
\newcommand{\m}{\mathfrak{m}}

\DeclareMathOperator{\ind}{ind}

\newcommand{\fol}{\mathcal{F}}
\newcommand{\SO}{\mathrm{SO}}

\newcommand{\Sym}{\operatorname{Sym}}

\newcommand{\Ad}{\operatorname{Ad}}

\title{Polar foliations on symmetric spaces and mean curvature flow}

\author[X. Liu]{Xiaobo Liu*}
\address[X. Liu]{Beijing International Center for Mathematical Research \& School of
Mathematical Sciences, Peking University, Beijing, China}
\email{xbliu@math.pku.edu.cn}
\thanks{*Research was partially supported by NSFC grants 11890662 and 11890660.}

\author[M. Radeschi]{Marco Radeschi**}
\address[M. Radeschi]{Department of Mathematics, University of Notre Dame, Notre Dame, IN, USA.}
\email{mradesch@nd.edu}
\thanks{**The author is partially supported by NSF grant 1810913}

\begin{document}

\begin{abstract}

In this paper, we study polar foliations on simply connected symmetric spaces with non-negative curvature. We will prove that all such foliations are isoparametric as defined in \cite{HeintzeLiuOlmos}. We will also prove a splitting theorem which, {}{when leaves are compact}, reduces the study of such foliations to polar foliations in compact simply connected symmetric spaces. Moreover, we will show that solutions to mean curvature flow of regular leaves in such foliations are always ancient solutions. This generalizes part of the results in \cite{TerngLiu-II} for mean curvature flows of isoparametric submanifolds in spheres.
\end{abstract}

\maketitle

\section{Introduction}

In this paper we consider polar foliations $(M,\fol)$ in a simply connected, non-negatively curved symmetric space $M$. Recall that {\it polar foliation} $\fol$ on a complete Riemannian manifold $M$ is a singular Riemannian foliation such that each point $x \in M$
is contained in a totally geodesic submanifold, called a {\it section}, which meets all leaves of $\fol$ and intersects them orthogonally. Polar foliations with flat sections are called {\it hyperpolar foliations}. Foliations given by orbits of polar actions by Lie groups are homogeneous examples of polar foliations. Other typical examples include the foliations by parallel and focal submanifolds of any isoparametric submanifold in a space form (cf. \cite{Terng-II}).  Each equifocal submanifold in a compact symmetric space gives a hyperpolar foliation with leaves the images of parallel normal vector fields under {the} exponential map (cf. \cite{TerngThorbergsson}).

The study of isoparametric submanifolds can be traced back to Cartan's work on isoparametric hypersurfaces in 1930's. Such manifolds have become an important subject in submanifold geometry and have been extensively studied since then. A nice survey article on this subject can be found in \cite{Thorbergsson}.
For a general Riemannian manifold $M$, a submanifold $L$ in $M$ is called {\it isoparametric} if {the} normal bundle $\nu L$ is flat,
$\exp (\nu_p L)$ is totally geodesic in a neighbourhood of $p$ for every $p \in L$, and locally parallel submanifolds of $L$ have parallel mean curvature vector fields (cf. \cite{HeintzeLiuOlmos}). Here {\it parallel submanifolds} of $L$ mean images of parallel normal vector fields along $L$ under {the} exponential map. When $M$  is a space form, this notion coincides with Terng's definition of isoparametric submanifolds in \cite{Terng-II}. Equifocal submanifolds $L$ in a compact symmetric space defined by Terng and Thorbergsson in \cite{TerngThorbergsson} are precisely isoparametric submanifolds with $\exp (\nu_p L)$ flat in a neighbourhood of $p$ for every $p \in L$. The definition of an isoparametric submanifold $L$ given in \cite{HeintzeLiuOlmos}  is in purely local terms. In particular, one can not expect parallel submanifolds of $L$ to give a global foliation of the ambient space in general. In case that parallel submanifolds of $L$ do give a global foliation of the ambient space, such a foliation is called an {\it isoparametric foliation}. It turns out that each regular leaf of an isoparametric foliation is always an isoparametric submanifold (cf. Corollary 2.5 in \cite{HeintzeLiuOlmos}).

Polar foliations share many similar properties as isoparametric foliations. For example, Alexandrino and Toeben have proved in \cite{AlexandrinoToeben} that for polar foliations in a complete simply connected Riemannian manifold, each regular leaf has trivial normal holonomy. This implies that {the} normal bundle of each regular leaf is flat. The existence of sections for polar foliations also implies that $\exp (\nu_p L)$ is totally geodesic for all $p$ in any regular leaf $L$.
However, unlike in the isoparametric case, there is no restriction for the mean curvature of the leaves of polar foliations.

It is an interesting question when a polar foliation is indeed isoparametric. When the ambient manifold has negative sectional curvature, then in the compact case there are no nontrivial polar or isoparametric foliations (cf. \cite{Toeben0,Lytchaknegsec}), while in the simply connected case one can easily produce examples of polar foliations that are \emph{not} isoparametric (cf. the discussion in the first page of \cite{Toeben0}).

The first main result of this paper shows that the situation is entirely different when the symmetric space has non-negative curvature:

\begin{theorem}\label{T:parallel-H}
Every polar foliation $(M,\fol)$ on a simply connected symmetric space with non-negative curvature is isoparametric.
\end{theorem}

Although this will not be used in the sequel, we remark that Theorem \ref{T:parallel-H} implies that for such a foliation,
the mean curvature vector field along all regular leaves is \emph{basic} in the sense that it projects to a vector field on the manifold part of the leaf space $M/\fol$. It was proved in \cite{LytchakRadeschi} that, given a foliation with basic mean curvature vector field, there is an ``averaging operator'' projection $\operatorname{Av}: C^{\infty}(M)\to C^{\infty}(M)^\fol$ (where $C^{\infty}(M)^\fol$ denotes the algebra of smooth functions constant along the leaves of $\fol$) which commutes with Laplacian. This opens the possibility of studying polar foliations on symmetric spaces in terms of the algebra $C^\infty (M)^\fol$, together with the action of the Laplacian, as was done in \cite{MR, MR-II} for singular Riemannian foliations on spheres.

Splitting theorems play an import role in the classification of isoparametric and equifocal submanifolds (cf. \cite{Terng-II}, \cite{HeintzeLiu}, and \cite{Ewert}). These theorems assert that such submanifolds decompose into products of lower dimensional submanifolds if their associated Coxeter groups decompose. 
In \cite{Lytchak}, Lytchak proved that every polar foliation $(M,\fol)$ on a simply connected symmetric space with non-negative curvature
splits as product of hyperpolar foliations, polar foliations with spherical sections, and trivial foliations. Here a trivial foliation means the foliation given by fibers of the projection from a product of two manifolds to one of its components. In this paper, we will prove
a splitting theorem of another type.
 
\begin{theorem}\label{T:splitting}
Let $(M,\fol)$ be a polar foliation {with compact leaves} on a simply connected symmetric space with non-negative curvature. Then the foliation splits as the product of a polar foliation on the compact factor of $M$, and an isoparametric foliation on the Euclidean factor.
\end{theorem}

Theorem \ref{T:splitting} is the special case of a more general, yet slightly wordier, result (cf. {}{Theorem {\ref{T:total-splitting}}}).
Isoparametric foliations on Euclidean spaces have been completely classified (see, for example, survey articles \cite{Thorbergsson} and \cite{Chi}). Hence Theorem \ref{T:splitting} reduces the study of corresponding polar foliation to those in compact simply connected symmetric spaces. Note that canonical metrics on compact simply connected symmetric spaces have non-negative sectional curvature.

The mean curvature flow (abbreviated as {\it MCF}) of a submanifold $L$ in a Riemannian manifold $M$ is a map
$ f: I \times L \longrightarrow M $ satisfying
$$\frac{\partial f}{\partial t} = H(t, \cdot),$$
where $I$ is an interval and $H(t, \cdot)$ is the mean curvature vector field  of $L_t:=f(t, \cdot)$.
It was proved in \cite{TerngLiu} that the solution to MCF for any compact isoparametric submanifold in a Eucliean space or in a sphere
always exists over a finite interval [0, T) with each $L_t$ an isoparametric submanifold for $t \in [0, T)$ and it converges to a focal submanifold as $t$ goes to $T$. This result was generalized to MCF flow for equifocal submanifolds in \cite{Koike} and
MCF for regular leaves of an isoparametric foliation on a compact non-negatively curved space in \cite{AlexandrinoRadeschi-II}. 
It was also proved in \cite{AlexandrinoRadeschi-II} that such mean curvature flows always have type I singularity.
An immediate consequence of Theorem~\ref{T:parallel-H}, {}{Theorem~{\ref{T:total-splitting}}}, and results in \cite{AlexandrinoRadeschi-II}  is that the same result holds for MCF of regular leaves of any polar foliation on a simply connected symmetric space with non-negative curvature.

If a solution to MCF exists for all $t \in (-\infty, T)$ for some $T \geq 0$, then it is called an
{\it ancient solution}. Ancient solutions to MCF have been extensively studied in recent years since they are important in studying singularities of general MCF. So far most results about ancient solutions are for MCF in Euclidean spaces and spheres. We refer to
the reference in \cite{TerngLiu-II} for some of these results.
In \cite{TerngLiu-II}, it was proved that MCF for isoparametric submanifolds in Euclidean spaces and spheres always have ancient solutions. Moreover, in each isoparametric foliation on a sphere, there is a unique minimal regular leaf and MCF of any other regular leaves always converge to the unique minimal regular leave as $t$ goes to $-\infty$. Another main result of this paper is that MCF of regular leaves of any polar foliation on a simply connected symmetric space with non-negative curvature always have ancient solutions.

The main result, which applies in greater generality than symmetric space, is the following:

\begin{theorem}\label{T:MCF}
Let $(M,\fol)$ be an isoparametric foliation on a simply connected Riemannian manifold, with $M/\fol$ compact.

If $Ric_M(x)>Ric_\Sigma(x)$ for every regular point $p$ and every vector $x\in T_pM$ tangent to the section $\Sigma$, then there exists a unique minimal regular leaf $L_{\rm min}$. Furthermore, for any regular leaf $L$ in $\fol$ the solution of MCF $L_t$ with initial data $L_0=L$ is always an ancient solution converging to $L_{\rm min}$ as $t$ goes to $-\infty$.
\end{theorem}

As a corollary we get:
\begin{theorem}\label{T:ancient}
Let $(M,\fol)$ be a polar foliation on a compact simply connected, non-negatively curved symmetric space, with $M/\fol$ compact. Then for any regular leaf $L$ in $\fol$, the solution of MCF $L_t$ with initial data $L_0=L$ is always an ancient solution and it converges to a minimal leaf as $t\to -\infty$. Furthermore, if $(M,\fol)$ does not split a trivial factor $(M',\{pts.\})$ then the minimal leaf is unique.
\end{theorem}

This theorem will give many examples of ancient solutions of MCF in compact symmetric spaces. The proof of Theorem~\ref{T:MCF} is based on estimates of Jacobi fields using comparison theorem for solutions to the Riccati equation. This is completely different from the approach in \cite{TerngLiu-II} which relies on structure of Coxeter groups associated to isoparametric submanifolds and representations of mean curvature vectors in terms of curvature normals.

By Theorem \ref{T:total-splitting}, and the results in \cite{TerngLiu} and \cite{TerngLiu-II} one has a complete picture of the mean curvature flow with regular leaves as initial data in a polar foliation on complete simply connected symmetric spaces with non-negative curvature:

\begin{corollary}
Let $(M,\fol)$ be a polar foliation on a complete simply connected symmetric space with non-negative curvature. Then the solutions of the mean curvature flow starting at regular leaves of $\fol$ are ancient.
\end{corollary}

This paper is organized in the following way: In Section~\ref{sec:Pre}, we collect some known results about polar foliations and holonomy Jacobi fields which will be needed in the proof of above theorems. In Section~\ref{sec:hyper}, we prove a splitting result for hyperpolar foliations, i.e. Proposition~\ref{P:real-splitting}, which is the essential part of Theorem~\ref{T:splitting}. In Section~\ref{sec:spherical}, we study polar foliations with spherical sections and complete the proof of Theorems~\ref{T:parallel-H} and \ref{T:splitting}. Finally, we prove Theorems~\ref{T:MCF} and \ref{T:ancient} in Section~\ref{sec:MCF}.

\section{Preliminaries}
\label{sec:Pre}

\subsection{Decomposition theorem}

We will use in a fundamental way the following {}{decomposition theorem for} polar foliations by Lytchak (\cite{Lytchak}, Theorem 1.2):

\begin{theorem}[Decomposition theorem]\label{T:Decomposition}

Let $(M,\fol)$ be a polar foliation on a simply connected non-negatively curved symmetric space $M$. Then we have a splitting
\[
(M,\fol) = (M_{-1},\fol_{-1})\times (M_0,\fol_0) \times \prod_i (M_i,\fol_i)
\]
where:
\begin{enumerate}
\item $(M_{-1},\fol_{-1})$ is given by the fibers of the projection of $M_{-1}$ onto a direct factor.
\item $(M_0,\fol_0)$ is hyperpolar.
\item $(M_i,\fol_i)$ are polar foliations, whose section has constant positive sectional curvature (these were called \emph{spherical polar} in \cite{GroveZiller}).
\end{enumerate}
\end{theorem}
We will refer to the factors in the decomposition of $(M,\fol)$ as factors of type 1, 2, 3.

\subsection{Structure of polar foliations on simply connected manifolds}\label{SS:properties}

We collect here a number of results, about the structure of polar foliations on simply connected spaces.

Let $(M,\fol)$ be a polar foliation on a simply-connected space. Then:
\begin{enumerate}
\item The leaves of $\fol$ are closed, and the leaf space $M/\fol$ is a Hausdorff space (Theorem 1.2 of \cite{LytGeomRes}).
\item If there are singular leaves, the leaf space $M/\fol$ has boundary. Furthermore, {}{the points on the boundary correspond to} singular leaves, while points in the interior $(M/\fol)_0$ {}{correspond to} principal leaves (Theorem 1.6 of \cite{LytGeomRes}).
\item Given a section $\Sigma$, up to composing with the universal cover $\tilde{\Sigma}\to \Sigma$ we can assume that $\Sigma$ is simply connected and immersed in $M$. Then there is a discrete group $W$ of isometries of $\Sigma$ (called the \emph{Weyl group}) such that $\Sigma/W$ is isometric to $M/\fol$ (Proposition 4.16 of \cite{Toeben}). Furthermore, for $M$ simply connected, this group is generated by \emph{reflections}, i.e. isometries that fix a codimension 1 submanifold of $\Sigma$ called \emph{wall} (Theorem 1.1 of \cite{Alexandrino}).
\end{enumerate}

It follows that the leaf space is isometric to a smooth orbifold $\Sigma/\fol$, and away from its topological boundary it is a smooth convex manifold.

\subsection{Lagrangian families of Jacobi fields}\label{SS:Lagrangian}

We collect here the main definitions and results about Lagrangian families of Jacobi fields. The interested reader can find more information and proofs about the statements below, in \cite{LytJacobi} and \cite{radeschinotes}.

Let $\mathcal{V}$ be a vector bundle over an interval $I$, endowed with a Euclidean product $\langle\,,\,\rangle$ and a metric connection $\nabla$. A vector fields is then simply a function $X:I\to \V$ such that $X(t)\in \V_t$, and we will write $\nabla X(t)$ simply as $X'(t)$. Given a section $R\in \Sym^2(\V^*)$, a $R$-Jacobi field is a vector field $J:I\to \V$ such that $J''(t)+R_tJ(t)=0$ for $t\in I$.

A space $\Lambda$ of $R$-Jacobi fields is called \emph{isotropic} if
\[
\langle J_1'(t),J_2(t)\rangle-\langle J_1(t),J_2'(t)\rangle=0\qquad \forall J_1,J_2\in \Lambda, t\in I.
\]
Notice that the quantity is constant in $t$, so it is enough to check that is holds for some $t_0\in I$. An isotropic space of Jacobi fields is called \emph{Lagrangian} if furthermore $\dim \Lambda=\dim \V$.

Given an isotropic space of Jacobi fields $\Lambda$, the dimension of $\Lambda(t)=\{J(t)\mid J\in \Lambda\}$ is constant and equal to $\dim\Lambda$ for all but discretely many values $t_i$, where the dimension can drop. In this case, $t_i$ is called a \emph{focal distance} and the quantity $\dim \Lambda-\dim \Lambda(t_i)$ is the corresponding multiplicity. If $\Lambda(t)$ has maximal dimension, we say that $t$ is \emph{regular} otherwise it is \emph{singular}.

If $\Lambda$ is Lagrangian, then there exists a smooth family $S_t\in \Sym^2(\mathcal{V}_t^*)$ for all $t$, such that on regular times the equation $S_tJ(t)=J'(t)$ holds for all $J\in \Lambda$. Such operator satisfies the Riccati equation
\[
S_t'+S_t^2+R_t=0.
\]

Given an isotropic space $\Lambda$ of Jacobi fields along a geodesic $\gamma:\RR\to M$ and some interval $[a,b]$, let the \emph{index of $\Lambda$} over $[a,b]$ be
\[
\ind_{[a,b]}\Lambda=\sum_{t\in [a,b]}(\dim\Lambda-\dim\Lambda(t)).
\]
By the discussion above, the sum is actually finite for $[a,b]$ compact interval.

\subsection{Holonomy Jacobi fields in a polar foliation}\label{SS:holonomy}

Let $(M,\fol)$ be a polar foliation, and let $L_0$ be a principal leaf. Since the normal bundle is trivial and flat with respect to the normal connection, given a horizontal vector $x\in \nu_p L$ it is possible to extend $x$ to a parallel vector field $X$ along $L_0$, and this induces an \emph{end-point map}
\[
\phi_X:L_0\to M,\qquad \phi_X(q)=\exp_q X_q.
\]
The image of $\phi_X$ is the leaf through $\phi_X(p)$.

Fix a point $p\in L_0$. Then rescaling $X$ induces a family of maps $\phi_{tX}:L_0\times \RR\to M$ such that $\phi_{tX}(p)$ is the horizontal geodesic from $p$ with $\gamma'(0)=X(p)$, and for every $v\in T_pL_0$ the vector field $J_v(t):=d_p\phi_{tX}(v)$ is the Jacobi field along $\gamma(t)$ (called  \emph{holonomy Jacobi field}) with $J_v(0)=v$, $J_v'(0)=S_{\gamma'(0)}v$, where $S_{\gamma'(t)}$ denotes the {}{negative of the} shape operator of the leaves along $\gamma(t)$ in the direction $\gamma'(t)$.

Along $\gamma$, define $\V_t=\nu_{\gamma(t)} \Sigma$ with the Euclidean structure induced by the metric on $M$. Since $\Sigma$ is totally geodesic, $\V$ is parallel and in particular the Levi Civita connection restricts to a connection on $\V$. Letting $R_t\in \Sym^2(\V^*)$ be $R_t(v)=R(v,\gamma'(t))\gamma'(t)$, the $R$-Jacobi fields are simply the Jacobi fields in $M$ along $\gamma$, which stay in $\V$ the whole time.

Let $\Lambda_h$ denote the vector space spanned by holonomy Jacobi fields along $\gamma$. This can be seen as a Lagrangian space of $R$-vector fields in $\V$ along $\gamma$. For all regular times $t$, one has $\Lambda_h(t)=\V_t=T_{\gamma(t)}L_t$. Furthermore, the Riccati operator $S_t$ for $\Lambda_h$ coincides with the {}{negative of the} shape operator, $S_{\gamma'(t)}$.

\section{Factors of type 2: Hyperpolar foliations}
\label{sec:hyper}

In this section we focus our attention to factors of type 2, i.e. hyperpolar foliations $(M,\fol)$ on a simply connected symmetric space $M$ with non-negative curvature, without trivial factors.

The main goal is to prove Theorem \ref{T:splitting} for the factors of type 2.  That is, any factor of type 2 {}{with compact leaves} splits as a product of a hyperpolar foliation on a compact symmetric space, and an isoparametric foliation in Euclidean space.

We divide the section in three parts: First, given a polar foliation of type 2 $(M,\fol)$, we show that it splits as a product of foliations $(M_1,\fol_1)\times(M_2,\fol_2)$ such that the curvature operator on $M_1$ along $\fol_1$-horizontal directions is zero, and the curvature operator of $M_2$ along $\fol_2$-horizontal directions is only zero along the sections. Second, we show that $M_1$ is the Euclidean space. And finally, we show that $M_2$ is compact.

\subsection{Splitting of the foliation}
\begin{lemma}\label{L:condition}
Let $(M,\fol)$ be a factor of type 2, $p\in M$ a regular point, $\Sigma$ the section through $p$, $x\in T_p\Sigma$, and $\gamma(t)=\exp_p(tx)$ the corresponding horizontal geodesic. Finally, let $\V_t=\nu_{\gamma(t)}\Sigma$. The following are equivalent:
\begin{enumerate}
\item $\operatorname{tr}\big|_{\V_0}R(\cdot, x)x=0$.
\item $R(\cdot,x)x=0$.
\item $R_t=R(\cdot,\gamma'(t))\gamma'(t)=0$ for all $t$.
\item $\operatorname{tr}\big|_{\V_t}R_t=0$ for all $t$.
\item The space $\Lambda_h$ of holonomy Jacobi fields along $\gamma$ satisfies $\ind_{(-\infty,\infty)}\Lambda_h<\infty$.
\end{enumerate}
\end{lemma}
\begin{proof}
($1\Rightarrow 2$) Follows from the fact that the eigenvalues of $R(\cdot, x)x$ are non-negative, hence $R(v,x)x=0$ for $v\in \V_0$. But since $\Sigma$ is flat, one has that $R(y, x)x=0$ for $y\in T_p\Sigma$ as well.

($2\Rightarrow 3$) Follows from the fact that $R_t$ is parallel along $\gamma$ hence the eigenvalues of $R_t$ are constant along $\gamma$.

($3\Rightarrow 4$) and ($4\Rightarrow 1$) are obvious.

($3\Rightarrow 5$) Let $e_1,\ldots e_n\in \V_0$ be an orthonormal basis of eigenvectors for the Riccati operator $S_0$ of $\Lambda_h$ (cf. section \ref{SS:holonomy}), with eigenvalues $\mu_1\ldots \mu_n$. Then since $R_t=0$, the Jacobi fields in $\Lambda_h$ with $J_i(0)=e_i$ and $J_i'(0)=S_0e_i=\mu_ie_i$ are given by $J_i(t)=(1+\mu_it)E_i$. In particular, the $J_i(t)$ are everywhere orthogonal to one another, and the singular times for $\Lambda_h$ are $t_i=-{1\over\mu_i}$ whenever $\mu_i\neq 0$. In particular,  $\ind_{(-\infty,\infty)}\Lambda_h\leq n<\infty$.

($5\Rightarrow 4$) Suppose by contradiction that (4) does not hold, $\operatorname{tr}|_{\V_t}R_t>0$. Since the trace of $R_t$ is constant along $\gamma$, it follows that $\operatorname{tr}|_{\V_t}R_t>n\delta>0$ for some $\delta$. Fix a regular point $q=\gamma(t_*)$ along $\gamma$, and consider the function $a(t)={1\over n}\operatorname{tr}|_{\V_{t+t_*}} S_{t+t_*}$, where $S_t$ is as usual the Riccati operator $S_t$ of $\Lambda_h$. Since $\operatorname{tr}|_{\V_t}R_t>n\delta>0$ we can apply the Average Comparison Theory for the Riccati operator, to obtain that $a(t)\leq \bar{a}(t)$ where $\bar{a}(t)$ is the solution of the model equation $\bar{a}'+\bar{a}^2+\delta=0$, with initial condition $\bar{a}(0)=a(0)$. Such a solution is given by $\bar{a}(t)=\sqrt{\delta}\tan(\sqrt{\delta}(t_0-t))$ for some $t_0$. As a consequence of the Comparison Theorem, it follows that the first positive singular time of $\Lambda_h$, which coincides with the first time $t_1$ such that $\lim_{t\to t_1^-}a(t)=-\infty$, is bounded above by ${\pi\over \sqrt{\delta}}$. That is, any two singular times of $\Lambda_h$ are at most ${\pi\over \sqrt{\delta}}$ apart. Since every singular time contributes at least 1 to the index, it follows that $\ind_{(-\infty,\infty)}\Lambda_h=\infty$.
\end{proof}

Given a type 2 factor $(M,\fol)$, let $\Sigma$ be a section. Again, we will think of $\Sigma\simeq \RR^n$ as a flat space, (possibly not injectively) immersed in $M$.

For each $p\in \Sigma$, define $R: T_p\Sigma\to \Sym^2(T_pM)$ given by $x\mapsto R(\cdot, x)x$, and let $\mathcal{D}_p$ denote the kernel of $R$. Since $R$ maps $T_p\Sigma$ into the set of positive semidefinite self-adjoint endomorphisms of $T_pM$, it follows that $\mathcal{D}_p$ is a vector space: in fact given $x_0,x_1\in \mathcal{D}_p$, let $x_t=tx_1+(1-t)x_0$ and $f(t)=\operatorname{tr} R(\cdot, x_t)x_t$. Then $f(t)$ is a quadratic polynomial, everywhere nonnegative and equal to $0$ at $t=0,1$. Then $f(t)\equiv 0$ that is $x_t\in \mathcal{D}_p$ for every $t$.

\begin{lemma}\label{L:properties-of-D}
Let $(M,\fol)$ be a hyperpolar foliation on a simply connected symmetric space with nonnegative curvature. Given a section $\Sigma$, the distribution $\mathcal{D}\subseteq T\Sigma$ defined above is parallel (in particular integrable with totally geodesic integral manifolds), and contained in the Euclidean factor of $M$.
\end{lemma}
\begin{proof}
Let $\gamma$ be a path in $\Sigma$ and let $X(t)$ be a parallel vector field along $\gamma$ with $X(0)\in \mathcal{D}_{\gamma(0)}$ (hence $R(\cdot, X(0))X(0)=0$). Since $R$ is parallel, we then have that the $1$-form $R(\cdot, X(t))X(t)$ is parallel as well, and in particular zero everywhere since it is zero for $t=0$. Therefore $X(t)\in \mathcal{D}_{\gamma(t)}$ hence $\mathcal{D}$ is parallel.

Write now $M=G/H$ for some symmetric pair $(G,H)$ with $H$ compact. Furthermore, assume $eH=p$, and let $\pi:G\to M$ denote the canonical projection. Letting $\g,\h$ be the Lie algebras of $G$ and $H$ respectively, there is a splitting $\g=\h\oplus \m$ where $\m$ can be identified via $d_e\pi$ with $T_pM$. Recall that, with respect to this identification, the curvature operator of $M$ can be expressed as $R(x,y)z=[[x,y],z]$ for all $x,y,z\in \m$.
Let $x\in \mathcal{D}_p$. By Lemma \ref{L:condition}, for every  $v\in T_pM\simeq \m$ one has
\[
R(v,x)x=[[v,x],x]=0\Rightarrow \|[v,x]\|^2=-\langle[[v,x],x],v\rangle =0\Rightarrow [x,v]=0.
\]
Given $w\in \h$, we have that $y=[x,w]\in \m$ hence, using the bi-invariant metric in $\g$, we get
\[
\|[x,w]\|^2=\langle[x,w],[x,w]\rangle=-\langle[x,y],w\rangle=0
\]
and therefore $x$ belongs to the center of $\g$. In particular, there is a splitting $G=\RR^n\times G_c$ for some $n>0$ and $G_c$ some compact simply connected group (possibly $G_c=\{e\}$). Since $H$ is compact it is contained in $G_c$, hence $M=\RR^n\times G_c/H$, with $x$ contained in the Euclidean factor.
\end{proof}

\begin{lemma}\label{L:Weyl-splits}
Suppose that $(M,\fol)$ is a factor of type 2 whose distribution $\mathcal{D}$ is neither trivial nor it contains $T\Sigma$. Then the sections split as a product $\Sigma_1\times \Sigma_2$ with $\Sigma_1$ an integral manifold for $\mathcal{D}$, and the Weyl group $W$ splits as a product $W_1\times W_2$ where $W_i$ acts on $\Sigma_i$ and fixes $\Sigma_{2-i}$, $i=1,2$.
\end{lemma}
\begin{proof}
Let $\Sigma\simeq \RR^k$ be a (simply connected, immersed but possibly non-injectively) section of $\fol$, and denote by $\Gamma$ the set of codimension 1 affine subspaces of $\Sigma$ fixed by some reflection in the Weyl group $W$ (the \emph{walls} of $W$).

Fix a regular point $0\in \Sigma$ as the origin of $\Sigma$, denote with $\Sigma_1$ the integral submanifold of $\mathcal{D}$ through $0$, and denote $\Sigma_2$ the affine subspace of $\Sigma$ through $0$ perpendicular to $\Sigma_1$. By Proposition 3.6 in \cite{Terng}, the union of all walls for the Weyl group is precisely the set of singular points on $\Sigma$. Hence a geodesic starting at regular point in $\Sigma$ passes a wall if and only if there is an increase for the index of $\Lambda_h$. By Lemma \ref{L:condition},  a geodesic starting at a regular point in $\Sigma$ which is not tangent to $\Sigma_1$ must intersect infinitely many walls. In particular, the number of walls must be infinite.

We claim that $\Sigma_1$ intersects finitely many walls. In fact, assume that there is a sequence of walls $w_i$ intersecting $\Sigma_1$, and let $v_i$ be a unit normal vector for $w_i$. Then $v_i$ can be written as $a_i x_i + b_i y_i$ where $x_i$ and $y_i$ are unit vectors tangent to $\Sigma_1$ and $\Sigma_2$ respectively and $a_i \neq 0$. Without loss of generality, we can assume that the $x_i$'s converge to a unit vector $x$ tangent to $\Sigma_1$. For $i$ sufficiently large, $\langle v_i, x\rangle = \langle a_i x_i, x\rangle \neq 0$. Hence the geodesic $\exp(tx)$ intersects infinitely many walls $w_i$, which contradicts Lemma \ref{L:condition}.

Since a wall intersects $\Sigma_1$ if and only if its normal vector is not perpendicular to $\Sigma_1$, there are only finitely many walls whose normal vector can be written as $v=v_1+v_2$ with $v_i$ tangent to $\Sigma_i$ and $v_1 \neq 0$. For infinitely many other walls, their normal vectors must be tangent to $\Sigma_2$.

Since the action of the Weyl group preserves the set of walls, we claim that the normal vector to every wall is either tangent to $\Sigma_1$ or to $\Sigma_2$. In fact, assume by contradiction that there is a wall with normal vector $u=u_1+u_2$ and both $u_1,u_2$ are non-zero. Notice that, since a geodesic $\gamma(t)=\exp_ptu_2$ from a regular point $p$ in $\Sigma$ must intersect infinitely many walls $w_i$, which means that their normal vectors $v_i$ satisfy $\langle v_i,u_2\rangle\neq 0$. Furthermore, for all but finitely many of these the normal vector $v_i$ is tangent to $\Sigma_2$.

 The reflection $r$ through this wall will map a wall to another wall. It is easy to check that a reflection $r$ fixing a wall $w$ with unit normal $v$, takes a wall $w'$ with normal vector $v'$ to a wall with normal vector
\[
r_*(v')=v'-2\langle v',v\rangle v.
\]

Assuming that there is a wall with normal vector $u=u_1+u_2$, apply the corresponding reflection $r$ to the infinite walls $w_i$ above, whose normal vector $v_i$ is tangent to $\Sigma_2$ and such that  $\langle v_i,u_2\rangle\neq 0$. Each wall $r(w_i)$ has now normal vector $r_*(v_i)=v_i-2\langle v_i,u_2\rangle u$ and in particular its component tangent to $\Sigma_1$ is $-2\langle v_i,u_2\rangle u_1\neq 0$. Therefore, the infinitely many walls $r(w_i)$ intersect $\Sigma_1$, contradicting the fact that there are only finitely many such walls.

In particular $\Gamma=\Gamma_1\cup \Gamma_2$ where $\Gamma_i$ denotes the set of walls whose normal vector is tangent to $\Sigma_i$. By Lemma 2.4 in \cite{HeintzeLiu} the Weyl group splits as a product $W=W_1\times W_2$, where $W_i$ is generated by the reflections in $\Gamma_i$, and it acts on $\Sigma_i$ while fixing $\Sigma_{2-i}$ ($i=1,2$).
\end{proof}

The splitting of the Weyl group action induces a splitting of the symmetric space itself.
\begin{proposition}\label{P:splittingw}
Assume that $(M,\fol)$ is a factor of type 2, whose section splits $\Sigma=\Sigma_1\times \Sigma_2$ so that the Weyl group $W$ splits as $W=W_1\times W_2$, with $W_i$ acting on $\Sigma_i$ and fixing $\Sigma_{2-i}$. Then there is a splitting of the foliation $(M,\fol)=(M_1,\fol_1)\times (M_2,\fol_2)$ such that $(M_i,\fol_i)$ is a factor of type 2 with section $\Sigma_i$.
\end{proposition}

This result was proved by Ewert in \cite{Ewert} under the slightly stronger assumption that $M$ does not have Euclidean factors. {}{We wrote a proof of Proposition {\ref{P:splittingw}}, but recently a much shorter one has been given} by Silva and Speran\c{c}a {\cite{SS}} using the completeness of Wilking's foliation $(M,\fol^\#)$ dual to $(M,\fol)$. We will then omit our proof.

\subsection{The case $\mathcal{D}\supseteq T\Sigma$}
\begin{proposition}[If $\mathcal{D}\supseteq T\Sigma$]\label{P:Disall}

Suppose $(M,\fol)$ is a hyperpolar foliation on a simply connected symmetric space with non-negative curvature, let $\Sigma$ be a section and assume that the distribution $\mathcal{D}$ contains $T\Sigma$. Then there is a splitting $M=\RR^n\times M'$ such that $(M,\fol)$ splits as $(\RR^n,\fol_0)\times M'$. In particular, if $(M,\fol)$ is of type 2 (no non-trivial factors) then $\mathcal{D}\supseteq T\Sigma$ implies $M=\RR^n$.
\end{proposition}
\begin{proof}
Write $M=\RR^n\times (G_c/K)$ where $G_c/K$ is a symmetric space of compact type. Let $\Sigma$ be a section. By Lemma \ref{L:properties-of-D}, the distribution $\mathcal{D}$ of $\Sigma$ is everywhere tangent to the Euclidean factor of $M$. The assumption that $\mathcal{D}=T\Sigma$ implies that $\Sigma$ is contained in the Euclidean factor. We now claim that in fact \emph{every} section is contained in the Euclidean factor. Let $\Sigma'$ denote any other section. Given a regular point $p'\in \Sigma'$, let $L$ denote the leaf through $p'$ and let $p\in L\cap \Sigma$. Given $x'\in T_{p'}\Sigma'$, let $X$ the corresponding parallel vector field along $L$ and let $x=X(p)$. By equifocality, the geodesic $\exp_{p'}tx'$ meets the same singular leaves as $\exp_ptx$ at the same times, and in particular the familes $\Lambda_h$, $\Lambda'_h$ of holonomy Jacobi fields along $\exp_{p}tx$, $\exp_{p'}tx'$ respectively, satisfy $\ind_{(-\infty,\infty)}\Lambda_h'=\ind_{(-\infty,\infty)}\Lambda_h$. Since $x\in \mathcal{D}$ by assumption, it follows from Lemma \ref{L:condition} that the latter index is finite, hence $\ind_{(-\infty,\infty)}\Lambda_h'<\infty$ and thus $x'\in \mathcal{D}$ as well. Since $x'$ was arbitrary, $T\Sigma'\subseteq \mathcal{D}$ and therefore $\Sigma'$ is contained in $\RR^n$ as well.

Since the horizontal space of every leaf (regular or singular) at a point $p$ is spanned by the tangent spaces of the sections through $p$ (cf. \cite{Alex1}, Theorem 2.4(a)) it follows that for every point $p=(p_e,p_c)\in \RR^n\times G_c/K$, the leaf through $p$ contains $\{p_e\}\times G_c/K$, and thus $(\RR^n\times G_c/K,\fol)=(\RR^n, \fol_e)\times G_c/K$ as wanted.
\end{proof}

\subsection{Splitting of hyperpolar foliations}

We can now prove:

\begin{proposition}\label{P:real-splitting}
Let $(M,\fol)$ be a hyperpolar foliation on a simply connected symmetric space with nonnegative curvature.
Then there is a splitting
{\[(M,\fol)=(M_1,\fol_1)\times (\RR^k,\fol_{\RR^k}),\]} where $(M_1,\fol_1)$ satisfies $\operatorname{tr}|_{\V_p}R(\cdot,x)x>0$ for any horizontal vector $x\in T_pM_1$.
\end{proposition}
\begin{proof}
Let $(M,\fol)$ be a polar foliation on a simply connected symmetric space with non-negative curvature. It is enough to prove the statement assuming that there are no trivial factors. By Lemma \ref{L:Weyl-splits}, the section $\Sigma$ splits as a product $\Sigma_1\times \Sigma_2$ where $\operatorname{tr}|_{\V_p}R(\cdot, x_1)x_1>0$ for every $x_1\in T_p\Sigma_1$ and $R(\cdot, x_2)x_2=0$ for every $x_2\in T_p\Sigma_2$. By Proposition \ref{P:splittingw}, this induces a splitting $(M,\fol)=(M_1,\fol_1)\times (M_2,\fol_2)$ where $(M_i,\fol_i)$ is a polar foliation with section $\Sigma_i$.

{Since the horizontal spaces of $M_2$ are contained in $\mathcal{D}$, it follows that $M_2=\RR^k$ is a Euclidean space by Proposition {\ref{P:Disall}}. Finally, since every $x_1\in \Sigma_1$ satisfies  $\operatorname{tr}|_{\V_p}R(\cdot, x_1)x_1>0$, it follows by Lemma {\ref{L:condition}} that every horizontal geodesic in $\Sigma_1$ meets singular leaves infinitely often. By equifocality, every horizontal vector in $M_1$ meets singular leaves infinitely often, hence $\operatorname{tr}|_{\V_p}R(\cdot, x)x>0$ for \emph{every} point $p\in M_1$ and horizontal vector $x\in T_pM_1$.}
\end{proof}

As corollaries of Proposition \ref{P:real-splitting} we get:
\begin{corollary}\label{C:splitting}
Let $(M,\fol)$ be a hyperpolar foliation with compact leaves on a simply connected symmetric space with nonnegative curvature. Then there is a splitting $(M,\fol)=(M_1,\fol_1)\times (\RR^k,\fol_{\RR^k})$, where $M_1$ is compact.
\end{corollary}
\begin{proof}
By Proposition {\ref{P:real-splitting}} it is enough to prove that if $(M_1,\fol_1)$ has compact leaves and satisfies $\operatorname{tr}|_{\V_p}R(\cdot,x)x>0$ for any horizontal vector $x\in T_pM_1$, then $M_1$ is compact. Assume that $M_1=M_c\times\RR^m$ where $M_c$ is compact. Given a leaf $L\in \fol_1$, we prove by contradiction that the projection $\pi_e:L\to \RR^m$ is a submersion, thus proving that $m=0$. In fact, if not there exists a point $p = (p_c,p_e) \in L$ and $x_e \in T_p\RR^m$ perpendicular to $d_p\pi_e(T_pL)$. Then in particular $x = (0,x_e)\in T_{p_c}M_c\times T_{p_e}\RR^m$ is perpendicular to $L$, hence horizontal, but $R(\cdot,x)x=0$ contradicting the hypothesis.
\end{proof}

\begin{corollary}\label{C:compact quotient}
Given a hyperpolar foliation $(M,\fol)$ on a simply connected symmetric space with non-negative curvature, the leaf space $M/\fol$ is non-compact if and only if $(M,\fol)$ splits off a foliation $(\RR^k,\fol_{\RR^k})$ with more than one leaf.
\end{corollary}
\begin{proof}
If there are no trivial factors in $(M,\fol)$, let $(M,\fol)=(M_1,\fol_1)\times (\RR^k,\fol_{\RR^k})$ be the splitting of $(M,\fol)$ from Proposition \ref{P:real-splitting}. Then $M/\fol=M_1/\fol_1 \times \RR^k/\fol_{\RR^k}$ and it is enough to prove that $M_1/\fol_1$ is compact. This is equivalent to prove that for any regular point $p\in M_1$ and any horizontal direction $x$ from $p$, the geodesic $\exp_p(tx)$ meets the singular strata at some point. However, since $\operatorname{tr}|_{T_pL_p}(R(\cdot, x)x)>0$ by definition of $M_1$, it follows by Lemma \ref{L:condition} that $\exp_p(tx)$ meets the singular stratum at least once (in fact, infinitely many times).

Finally, assume that $(M,\fol)$ splits off trivial factors:
\[
(M,\fol)=N\times (N',\{pts.\})\times (M',\fol'),
\]
where $N'$ is foliated by points, $\fol|_N$ is one leaf, and $(M',\fol')$ has no trivial factors. We can split $N'$ further as $N'=N'_c\times \RR^{k'}$ with $N'_c$ compact, and thus
\[
(M,\fol)=N\times (N'_c,\{pts.\})\times (\RR^{k'},\{pts.\})\times (M',\fol')\quad \Rightarrow \quad M/\fol=N'_c\times \RR^{k'}\times M'/\fol'
\]
Thus $M/\fol$ is non-compact if and only if either $k'>0$ or $M'/\fol'$ is non-compact. By the discussion above, it follows that either way $(M,\fol)$ splits off a Euclidean factor $(\RR^k, \fol_{\RR^k})$ with more than one leaf.
\end{proof}

\section{Factors of type 3: spherical polar foliations}
\label{sec:spherical}

In this section, we focus our attention to factors of type 3. That is, a polar foliation $(M,\fol)$ on a simply connected symmetric space $M$, with sections of constant positive curvature. The main goal of this section is to prove that factors of type 3 are compact and isoparametric, i.e. they have parallel mean curvature vector field.

We start by proving compactness.
\begin{proposition}\label{P:type3cpt}
Let $(M,\fol)$ be a factor of type 3. Then $M$ is compact.
\end{proposition}
\begin{proof}
Since the leaf space has positive curvature, it is not a metric product and thus $(M, \fol)$ is indecomposable. By Corollary 6.5 in \cite{Lytchak}, such a foliation only has one dual leaf, which has to be $M$ itself. By Lemma 4.1 in \cite{Lytchak}, dual leaves have to be compact, so $M$ is compact. 
\end{proof}

Rescale the metric so that sections have positive sectional curvature 1, and we will consider the section as the (possibly non injectively) immersed round sphere. In particular, all horizontal geodesics are closed with common (not necessarily smallest) period $2\pi$, and the end-point map $\phi_{tX}$ has period $2\pi$ whenever $X$ is a parallel normal vector of unit length along a regular leaf.

Finally, since $M$ is a symmetric space of non-negative curvature, along any horizontal geodesic $\gamma(t)$ the eigenvalues of $R_t$ are constant and non-negative, and we call them $0=\lambda_0\leq\lambda_1^2\leq\ldots\leq\lambda_n^2$.

\begin{lemma}\label{L:integers}
Given a factor $(M,\fol)$ of type 3, with metric rescaled so that the section $\Sigma$ has sectional curvature $1$. Then fixing a regular point $p$ in $\Sigma$ and a unit-speed horizontal geodesic $\gamma(t)=\exp_ptx$, the eigenvalues of the curvature operator $R_t$ along $\gamma$ are squares of integers. Furthermore, the kernel of $R_t$ is contained in the kernel of the shape operator $S_{\gamma'(t)}$.
\end{lemma}
\begin{proof}
Recall that a holonomy Jacobi field $J(t)$ along $\gamma$ is given by $d_p\phi_{tX}(v)$ for some $v\in \mathcal{V}_p$. In particular, since $\phi_{tX}$ is periodic with period $2\pi$, so is any holonomy Jacobi field.

For any eigenvalue $\lambda^2$ of $R_t$, let $\V^\lambda_t\subseteq \V_t$ be the corresponding eigenspace. Since $R_t$ is parallel, so are the sub-bundles $\V^\lambda$. Furthermore, since $R_{2\pi}=R_0$, then $\V^\lambda_{2\pi}=\V^\lambda_0$ and in particular the projection $\pi_{\V^\lambda}J$ of a holonomy Jacobi field $J$ onto a subspace $\V^\lambda$ is again a periodic Jacobi field.

Recall however that $\pi_{\V^\lambda}J(t)=\sin(\lambda t)V(t)+\cos(\lambda t)W(t)$ for some parallel vector fields $V(t), W(t)\in \V^\lambda_t$. Since $\pi_{\V^\lambda}J$ is periodic with period $2\pi$, it follows that each $\lambda$ is an integer.

Finally, since holonomy Jacobi fields are periodic, they must be of the form
\[
J(t)=V_0(t)+\sum_i\left(\sin(\lambda_i t)V_i(t)+\cos(\lambda_i t) W_i(t)\right)
\]
for $V_0(t)$ parallel and tangent to $\V^0$, and $V_i(t), W_i(t)$ parallel and tangent to $\V^{\lambda_i}$. In particular $S_tJ=J'(t)\perp \V^0$, i.e. $S_t|_{\V^0}=0$.
\end{proof}

\begin{lemma}\label{L:sum}
Let $\Lambda_{h}$ be the space of holonomy Jacobi fields along a horizontal geodesic $\gamma$, and let $m=\ind_{[0,2\pi)}\Lambda_{h}$. This index does not depend on the choice of initial regular point $p$ or horizontal geodesic $\gamma$, and it equals $2\sum_i \lambda_i$.
\end{lemma}

\begin{proof}
The independence on the choice of geodesic, or on the regular point $p$, follows from the continuity of the index for Lagrangian spaces of Jacobi fields proved in Proposition 1.4 of \cite{LytJacobi}, and the fact that in this case, any two horizontal closed geodesics of the same length can be connected via a path of horizontal closed geodesics of constant length (because the sections are round spheres).

For the second statement,  fix a horizontal geodesic $\gamma$ and notice that for any integer $k$, $\ind_{[0,2\pi k)}\Lambda_{h}=mk$. We now consider a different Lagrangian space of Jacobi fields in $\mathcal{V}$ along $\gamma$, namely
\[
\Lambda_0=\{J\mid J(0)=0, J'(0)\in \mathcal{V}_{\gamma(0)}\}.
\]
Let $\{e_1,\ldots e_n\}\subseteq \V_0$ be a basis of eigenvalues of $R_0$. Notice that any Jacobi field $J_i\in \Lambda_0$ with $J_i(0)=0$, $J_i'(0)=e_i$ is given by $J_i(t)=\sin(\lambda_i t)E_i(t)$ (with $E_i(t)$ the parallel extension of $e_i$), which vanishes $2\lambda_i$ times for every period $2\pi$ of the geodesic. Since $\{J_1(t),\ldots J_n(t)\}$ are linearly independent whenever they are all nonzero, it follows that:
\[
\ind_{[0,2\pi)}\Lambda_0=2\sum_i\lambda_i\qquad\Rightarrow\qquad \ind_{[0,2k\pi)}\Lambda_0=2k\sum_i\lambda_i .
\]
On the other hand, it follows by Proposition 1.4 of \cite{LytJacobi} that given two Lagrangian spaces of Jacobi fields $\Lambda_1,\Lambda_2$ along $\mathcal{V}$, then for \emph{any} interval $I$ one has:
\[
|\ind_I\Lambda_1-\ind_I\Lambda_2|\leq \dim \mathcal{V} .
\]
Applying this to the case of $\Lambda_h$ and $\Lambda_0$, one has that for any positive integer $k$, $k|m-2\sum_i\lambda_i|<n$. The only way this can be true for all $k$ is that $m-2\sum_i\lambda_i=0$.
\end{proof}

\begin{remark}
It follows from the previous lemma that $m$ is even, but this should not be a surprise. Consider, in fact, the section $\Sigma=S^n$, cut by the walls $w_i$ fixed by reflections in the Weyl group. Given a horizontal geodesic $\gamma$ along $\Sigma$, the singular times for the family $\Lambda_h$ of holonomy Jacobi fields along $\gamma$, i.e. the times in which $\dim\Lambda_h(t_i)<\dim\Lambda_h$ (cf. Section \ref{SS:Lagrangian}), coincide with the times in which $\gamma(t_i)$ meets a wall $w_i$ of $\Sigma$. Furthermore, the multiplicity $m_i= \dim\Lambda_h-\dim\Lambda_h(t_i)$ corresponds to the multiplicity of the wall $w_i$. Since $\gamma$ meets each wall $w_i$ \emph{twice}, at times $t_i$ and $t_i+\pi$, each wall contributes $2m_i$ to the index $\ind_{[0,2\pi)}\Lambda_h$.
\end{remark}
\begin{corollary}\label{L:cond-satisf}
Given a horizontal geodesic $\gamma$, the curvature operator $R_t$ satisfies $\operatorname{tr}|_{\V_t}R_t>0$.
\end{corollary}
\begin{proof}
It follows by Lemma 2.3 of \cite{Lytchak} that a polar foliation on a simply connected, non-negatively curved symmetric space must contain singular leaves, unless the foliation is of type 2. Therefore we have at least one singular leaf $L$, and there is at least one horizontal closed geodesic $\gamma$ through $L$, which must satisfy $\ind_{[0,2\pi)} \Lambda_h=:m>0$.

By Lemma \ref{L:sum}, this index is the same for \emph{any} horizontal closed geodesic in the section. Since $\sum_i \lambda_i={m\over 2}>0$ by the previous lemma, then $\operatorname{tr}|_{\V_t}R_t=\sum_i\lambda_i^2>0$.
\end{proof}

\begin{lemma}\label{L:sumi}
Fix a factor $(M,\fol)$ of type 3, with the metric normalized as above. Then for any horizontal geodesic $\gamma(t)$, the function $\det(d\phi_{tX})$ can be written as a linear combination
\[
f(t)=\sum_i a_i\sin(s_it)+b_i\cos(s_it),
\]
where $s_i$ are integers of the same parity between $-{m\over 2}$ and $m\over 2$, where ${m\over 2}=\ind_{[0,\pi)}\Lambda_h$.
\end{lemma}
\begin{proof}
As usual let $\mathcal{V}_t$ be the bundle along $\gamma$ perpendicular to $\Sigma$. Let $e_1\ldots e_n$ be an orthonormal basis of eigenvectors of $R_0$, with eigenvalues $\lambda_i^2$, and let $E_i(t)$ be the parallel extension of $e_i$ along $\gamma$.  Assume that $\lambda_i=0$ for $i=1,\ldots r$ and $\lambda_i>0$ for $i=r+1,\ldots n$. Let $\Lambda_h$ be the family of holonomy Jacobi fields along $\gamma$, with a basis $J_1,\ldots J_n$ with
\[
J_i(0)=e_i,\qquad J_i'(0)=\sum_jb_{ij}e_j.
\]
Since $M$ is a symmetric space, $R_t$ is parallel and the Jacobi fields can be explicitly computed as
\[
J_i(t)=\sum_{j=1}^r\left(\delta_{ij}+b_{ij}t\right)E_j(t) + \sum_{j=r+1}^n\left(\delta_{ij}\cos\lambda_jt+{b_{ij}\over \lambda_j}\sin\lambda_jt\right)E_j(t).
\]
From Lemma \ref{L:integers} one has $J_i(t)=E_i(t)$ for $i=1,\ldots r$, and for $i=r+1,\ldots n$
\[
J_i(t)=\sum_{j=r+1}^n\left(\delta_{ij}\cos\lambda_jt+{b_{ij}\over \lambda_j}\sin\lambda_jt\right)E_j(t) .
\]

Since $\Sigma=S^n$ is simply connected, the normal holonomy of $\Sigma$ is contained in $\SO(n)$ and in particular $E_1(t),\ldots E_n(t)$ represent at each point an oriented orthonormal basis of $\V_t$. Hence, it makes sense to define $f(t):=\det(d_p\phi_{tX})=\det(\langle J_i,E_j\rangle)$, which is given by a linear combination of terms of the form
\begin{equation}\label{E:functions}
\prod_{i=r+1}^{n}\operatorname{sc}_i(\lambda_it),\qquad \operatorname{sc}_i\in \{\sin,\cos\} .
\end{equation}

Using the product formulas for trigonometric functions, it follows that $f(t)$ is a linear combination
\[
f(t)=\sum_i a_i\sin(s_it)+b_i\cos(s_it),
\]
where each $s_i$ is a linear combination $s_i=\epsilon_{r+1}\lambda_{r+1}+\ldots +\epsilon_n \lambda_n$ with coefficients $\epsilon_r,\ldots \epsilon_n\in \{\pm 1\}$. In particular, $s_i$ are integers, bounded between $\sum\lambda_i={m\over 2}$ and $-{m/2}$. All $s_i$ have the same parity, since their difference is a linear combination of the $\lambda_i$'s with coefficients in $\{-2,0,2\}$.
\end{proof}

\begin{proposition}\label{P:type3-gen-isoparametric}
Let $(M,\fol)$ be a factor of type 3. Then the mean curvature is basic.
\end{proposition}
\begin{proof}
Fix a regular leaf $L_0$ and a basic horizontal vector field $X$ along $L_0$. For $p\in L_0$, let $\Sigma_p$ be the section through $p$, $\gamma_p(t)=\exp_ptX_p$, $\mathcal{V}_p=\nu \Sigma_p|_{\gamma_p}$, $\Lambda_p$ the space of holonomy Jacobi fields along $\gamma_p$, and $E_1(t),\ldots E_n(t)$ a frame of parallel vector fields along $\gamma_p$, tangent to $\mathcal{V}_p$. Finally, let $f_p(t)=\det(\langle J_i(t),E_j(t)\rangle)$.

Once again, we normalize the metric so that the section is a round sphere of curvature 1. By Lemma \ref{L:integers} the eigenvalues $\lambda_i^2$ of $R_t$ are squares of integers, and by Lemma \ref{L:sum},  $\sum_i\lambda_i={m\over 2}=\ind_{[0,\pi)}\Lambda_{L_0}$. Furthermore, by Lemma \ref{L:sumi},
\[
f_p(t)=\sum_i a_i\sin(s_it)+b_i\cos(s_it),\qquad s_i=\lambda_1\pm \lambda_2\ldots \pm \lambda_n
\]
where $s_i$ can range within integers of the same parity from $-\sum_i \lambda_i=-{m\over 2}$ to ${m\over 2}$. Taking into account that $\cos(s_it)=\cos(-s_it)$, $\sin(-s_it)=-\sin(s_it)$ and $\sin(0)=0$, it follows that depending on the parity of ${m\over 2}$, the functions $\sin(s_it),\cos(s_it)$ are contained in the space $\mathcal{T}$ of functions:
\[
\left\{\begin{array}{ll}
\mathcal{T}=\operatorname{span}\{1, \cos(2t), \sin(2t), \cos(4t), \sin(4t),\ldots, \cos({m\over 2}t), \sin({m\over 2}t)\}&\textrm{if ${m\over 2}$ is even}\\
\mathcal{T}=\operatorname{span}\{\cos(t), \sin(t), \cos(3t), \sin(3t),\ldots, \cos({m\over 2}t), \sin({m\over 2}t)\}&\textrm{if ${m\over 2}$ is odd}
\end{array}
\right.
\]
in either case of dimension ${m\over 2}+1$, which does not depend on $p$. The projection of $\gamma_p$ to $M/\fol$ will intersect singular strata at singular times $t_1,\ldots t_k\in (0,\pi)$, and for each $j=1,\ldots k$ we can let $m_j:=\dim L_0-\dim L_{t_j}=\dim \Lambda_p-\Lambda_p(t_j)$. Notice that, by the equifocality of singular Riemannian foliations (cf. Proposition 4.3 of \cite{LT},  Theorem 2.9 in \cite{Ale10} or Proposition 2.26 of \cite{radeschinotes}) the data $t_j$, $m_j$ do not depend on the choice of point $p\in L_0$ but only on the choice of basic vector field $X$: in fact, if we chose a different point $q$ and let $\gamma_q(t)=\exp_q(tX_q)$ then $\gamma_p$ and $\gamma_q$ would meet the same leaf at each time $t$.

The fact that $\gamma_p(t)$ meets the singular leaves $L_{t_j}$ with $\dim\Lambda_p-\dim\Lambda_p(t_j)=m_j$ can be restated by saying that $f_p(t)$ vanishes with order $m_j$. This imposes, for every singular time $t_j$, exactly $m_j$ conditions:
\[
f_p(t_j)=f_p'(t_j)=\ldots=f_p^{(m_j-1)}(t_j)=0.
\]

These conditions form a system of $\sum_{j=1}^km_j={m\over 2}$ linear equations on $\mathcal{T}$, which are easily seen to be linearly independent:
in fact, consider the subspace $\mathcal{T}'\subseteq \mathcal{T}$ spanned by the $m\over 2$ linearly independent functions:
\[
{\cos^l(t-t_j)\over \sin^l(t-t_j)}\prod_{i=1}^{k}\sin^{m_i}(t-t_i),\qquad j=1,\ldots k, l=1,\ldots m_j.
\]
 Then the linear map $\mathcal{T}\to \RR^{m\over 2}$ which sends a function $h(t)\in \mathcal{T}$ to
 \[
 (h^{(l)}(t_j))_{l=0,\ldots,m_j-1,\,j=1,\ldots k}
 \]
 is invertible when restricted to $\mathcal{T}'$ (as the matrix for this map is triangular with non-zero diagonals with respect to that basis). Hence the kernel of this map has dimension 1, and $f_p$ is the unique function in the kernel satisfying $f_p(0)=1$.

In particular, $f_p(t)$ is uniquely determined by $X$ from information on the leaf space $M/\fol$, and it is independent of $p$. Since
\[
f_p'(0)=f_p(0)\cdot \operatorname{tr}{S_{\gamma_p'(0)}}=-\langle H_p, X_p\rangle,
\]
it follows in particular that the inner product of $H$ with any basic horizontal vector field $X$ along $L_0$ is constant. Thus, $H$ is basic as well.
\end{proof}

\begin{remark}
By Theorem 1.2 in \cite{Lytchak}, a polar foliation whose section has constant positive sectional curvature, is either a foliation of codimension 2, or a polar foliation on a sphere or projective space $(S^n,\fol)$, $(\CP^n,\fol)$, $(\HP^n,\fol)$. In the latter case, the foliation can always be lifted to a polar foliation on a round sphere $(S^n,\fol)$, where the fact that the mean curvature vector field is basic was proved by Alexandrino and the second author in \cite{Alexandrino-Radeschi}. In particular, the new information of Theorem \ref{T:parallel-H} applies to polar foliations with 2-dimensional round section.
\end{remark}

\begin{remark}
Even though the mean curvature vector field of factor of type 3 is basic, the second fundamental form of leaves need not be constant in any natural sense, even in the case of isoparametric hypersurfaces in $\CP^n$, cf. \cite{Park}.
\end{remark}

We end this section with a proof of Theorems \ref{T:parallel-H} (that is, polar foliations on symmetric spaces with non negative curvature are isoparametric) and \ref{T:splitting} (polar foliations with compact leaves on symmetric spaces with nonnegative curvature split into a compact factor and an Euclidean one).

\begin{proof}[Proof of Theorem \ref{T:parallel-H}]
We check that this is true on every factor of the foliation.

This is trivially true for factors of type 1. For factors $(M_0,\fol_0)$ of type 2 (hyperpolar foliations) this fact  follows from
Theorems 2.4 and 6.5 in \cite{HeintzeLiuOlmos}. Finally, we proved that factors of type 3 are isoparametric in Proposition \ref{P:type3-gen-isoparametric}.
\end{proof}

\begin{proof}[Proof of Theorem \ref{T:splitting}]
Let $(M,\fol)=(M_{-1},\fol_{-1})\times (M_0,\fol_0)\times \prod_i (M_i,\fol_i)$ be Lytchak's decomposition (cf. Theorem \ref{T:Decomposition}). The result follows because each factor decomposes accordingly: For the trivial foliation $(M_{-1},\fol_{-1})$ the result is obvious. For the type 2 factor $(M_0,\fol_0)$ the result follows by Corollary \ref{C:splitting}. For each type 3 factor $(M_i,\fol_i)$ the result follows by Proposition \ref{P:type3cpt}.
\end{proof}

The following is a stronger version of Theorem \ref{T:splitting}:

\begin{theorem}\label{T:total-splitting}
Let $(M, \fol)$ be a polar foliation on a simply connected symmetric space with non-negative curvature. Then there is a unique splitting $(M,\fol)=(M_c, \fol_c)\times (M_e,\fol_e)$ where $M_c/\fol_c$ is compact, $M_e$ is a Euclidean space, and $(M_e,\fol_e)$ is an isoparametric foliation with compact leaves.
\end{theorem}

\begin{proof}
Again let $(M,\fol)=(M_{-1},\fol_{-1})\times (M_0,\fol_0)\times \prod_i (M_i,\fol_i)$ be Lytchak's decomposition (cf. Theorem \ref{T:Decomposition}). Let
\begin{enumerate}
\item $(M_{-1},\fol_-)=M^a_{-1}\times (M_{-1}^b,\{pts\})$ be the splitting of $M_{-1}$ into the factor foliated by points, and the factor with one leaf only. Furthermore, split $M_{-1}^b=M_{-1}^c\times M_{-1}^e$ with $M_{-1}^c$ compact and $M_{-1}^e$ Euclidean.
\item $(M_0,\fol_0)=(M_0^c,\fol_0^c)\times(M_0^e,\fol_0^e)$ be the splitting of $\fol_0$ from Proposition \ref{P:real-splitting}.
\item $(M^c,\fol^c)$ be the product foliation $(M^c,\fol^c)=M_{-1}^a\times (M_{-1}^c,\{pts.\})\times (M_0^c,\fol_0^c)\times \prod_i(M_i,\fol_i)$.
\item $(M^e,\fol^e)=(M_{-1}^e,\{pts.\})\times (M_0^e,\fol_0^e)$.
\end{enumerate}
Then $(M,\fol)$ splits as a product $(M,\fol)=(M^c,\fol^c)\times (M^e,\fol^e)$, where:
\begin{itemize}
\item $M^c/\fol^c=M_{-1}^c\times M_0^c/\fol_0^c\times \prod_i M_i/\fol_i$ is compact by Corollary \ref{C:compact quotient}, and the fact that $M_i/\fol_i$ are spherical quotients.
\item $(M_0^e,\fol_0^e)$ is an isoparametric foliation in Euclidean space without trivial factors, and thus its leaves are compact. Therefore, $(M^e,\fol^e)=(M_{-1}^e,\{pts.\})\times (M_0^e,\fol_0^e)$ is an isoparametric foliation in a Euclidean space, with compact leaves.
\end{itemize}
\end{proof}

\section{Minimal isoparametric leaves and mean curvature flow}
\label{sec:MCF}

\subsection{Minimal isoparametric leaves in positive Ricci curvature}
The following is a generalization of a well-known result for families of parallel hypersurfaces in spaces with positive Ricci curvature {}{(cf. {\cite{VerdianiZiller}} for results that generalize this to intermediate Ricci curvature):}

\begin{proposition}\label{P:volume-concavity}
Let $(M,\fol)$ be a polar foliation of dimension $n$ with compact leaves on a simply connected manifold $M$, with projection $\pi:M\to M/\fol$. Assume that for any principal leaf $L$ and any $x\in\nu_pL$, the curvature operator $R$ on $M$ satisfies $\operatorname{tr}_{T_pL} R(\cdot, x)x>0$. {}{Let $\operatorname{vol}$ denote the $n$-dimensional volume}. Then the function
\[
V:M/\fol\to \RR\qquad V(p_*)=\operatorname{vol}\big(\pi^{-1}(p_*)\big)^{1\over n}
\]
is strictly concave on the regular part of $M/\fol$, and equal to $0$ on the singular part. In particular, if $M/\fol$ is compact, there is a unique leaf achieving the maximum volume.
\end{proposition}

\begin{proof}
Let $\gamma_*:[-a,b]\to M/\fol$ a geodesic segment on the regular part of $M/\fol$. It is enough to prove that $V(\gamma_*(t))$ is strictly concave.

Let $L_t=\pi^{-1}(\gamma_*(t))$, and let $X$ the horizontal parallel vector field along $L_0$ projecting to $\gamma_*'(0)$, and let $\phi_{tX}:L_0\times[-a,b]\to M$ the end-point map defined in Section \ref{SS:holonomy}. Then for every $p\in L_0$, $\gamma_p(t)=\phi_{tX}(p)$ is a horizontal geodesic in $M$ projecting to $\gamma_*$. Letting $\omega_t$ be the volume form of $L_t$, one has that $\phi_{tX}^*\omega_t(p)=f(p,t)\omega_0$ for some function $f$. Here the function $f(p,t)$ can be computed as
\[
f(p,t)=\det (d_p\phi_tX)=\det(\langle J_i(t),E_j(t)\rangle)
\]
where $J_i(t)$ are holonomy Jacobi fields, $E_i(t)$ are parallel vector fields, and $J_i(0)=E_i(0)=e_i$ is an oriented, orthonormal basis of $T_pL_0$. Then
\[
V(\gamma_*(t))=\operatorname{vol}(L_t)^{1\over n}=\left(\int_{L_t}\omega_t\right)^{1\over n}=\left(\int_{L_0}f(p,t)\omega_0\right)^{1\over n}  .
\]

Fixing a $p\in L_0$, recall from Section \ref{SS:holonomy} that the holonomy Jacobi fields form a Lagrangian space $\Lambda_p$ of Jacobi fields of the bundle $\V$ given by $\V_t=T_{\gamma_p(t)}L_t$. In particular there is a Riccati operator $S\in \Sym^2(\V^*)$ along $\gamma_p(t)$ such that $S_tJ_i(t)=J_i'(t)$, which solves the ODE $S_t'+S_t^2+R_t=0$. Since by assumption $\operatorname{tr}|_{\V}R_t>0$, let $\delta>0$ be such that along $\gamma_p$, $\operatorname{tr}|_{\V}R_t>n\delta$. {Up to rescaling the metric on $M$, we can assume that $\delta=1$ and $\operatorname{tr}|_{\V}R_t>n$.}

By comparison theory of the Riccati operator, letting $s_0={1\over n}\operatorname{tr}(S_0)$, one has that $s(t):={1\over n}\operatorname{tr}(S_t)$ is bounded above by the solution $\bar{s}(t)$ of the ODE
\[
\left\{\begin{array}{l}
\bar{s}'(t)+\bar{s}^2(t)+1=0\\
\bar{s}(0)=s_0
\end{array}
\right.
\]
that is, $\bar{s}(t)=-\tan(t-t_0)$, where $t_0=\arctan(s_0)$.
Finally, $f(p,t)$ satisfies ${d\over dt}(\ln f(p,t))=\operatorname{tr}(S_t)\leq-n\tan(t-t_0)$. Hence for any $t>0$
\[
\ln\left({f(p,t)\over f(p,0)}\right)\leq \int_0^t-n\tan(t-t_0)dt=\ln \left({\cos^n(t-t_0)\over \cos^n(t_0)}\right) .
\]
Since $f(p,0)=1$,
\[
f(p,t)\leq{\cos^n(t-t_0)\over \cos^n(t_0)},\qquad t>0.
\]
For negative values of $t$, we can repeat the same argument for $\hat{\gamma}(t):=\gamma(-t)$. In this case, $\operatorname{tr}(S_{\hat{\gamma}'(0)})=-\operatorname{tr}(S_{{\gamma}'(0)})=-s_0$, and one can apply the comparison theory to obtain ${1\over n}\operatorname{tr}(S_{\hat{\gamma}(t)})\leq \hat{{s}}(t)$ where $\hat{{s}}(t)$ now solves
\[
\left\{\begin{array}{l}
\hat{{s}}'(t)+\hat{{s}}^2(t)+1=0\\
\hat{{s}}(0)=-s_0
\end{array}
\right.
\]
that is, $\hat{{s}}(t)=-\tan(t+t_0)$. Now, $\hat{f}(p,t):=f(p,-t)$ solves the ODE
\[
{d\over dt}(\ln \hat{f}(p,t))=\operatorname{tr}(S_{\hat{\gamma}(t)})\leq-n\tan(t+t_0)
\]
and again since $\hat{f}(p,0)=1$, one obtains for any $t>0$
\[
\hat{f}(p,t)\leq{\cos^n(t+t_0)\over \cos^n(-t_0)} .
\]
Substituting $\hat{f}(p,t)=f(p,-t)$ one gets, now for negative values of $t$, that
\[
{f}(p,t)\leq\left({\cos(-t+t_0)\over \cos(-t_0)}\right)^n=\left({\cos(t-t_0)\over \cos(t_0)}\right)^n
\]
Therefore, the same inequality for $f(p,t)$ applies to both sides of $t=0$. In particular, we have
\begin{align*}
V(\gamma_*(t))=&\left(\int_{L_0}f(p,t)\omega_0\right)^{1/n}\leq {\cos(t-t_0)\over \cos(t_0)}\left(\int_{L_0}\omega_0\right)^{1/n}\\
=&{\cos(t-t_0)\over \cos(t_0)}V(\gamma_*(0))
\end{align*}
with equality at $t=0$. In particular,
\[
{d^2\over dt^2}\Big|_{t=0}V(\gamma_*(t))\leq {d^2\over dt^2}\Big|_{t=0}\left({\cos(t-t_0)\over \cos(t_0)}V(\gamma_*(0))\right)=- V(\gamma_*(0))<0.
\]

Hence $V$ is strictly concave in the interior of $M/\fol$. Since points on the boundary of $M/\fol$ corresponds to lower dimensional leaves, $V$ is $0$ on the boundary.  Moreover if $M/\fol$ is compact,  $V$ must have a maximum in the interior and this is the only critical point in the interior since  the interior of $M/\fol$ is convex (cf. Section \ref{SS:properties}).
\end{proof}

\begin{remark}\label{R:noncompact-case}
The setup in the statement of Proposition \ref{P:volume-concavity} can be extended to the case of noncompact leaves, as follows. Let $(M,\fol)$ be a polar foliation with non-compact leaves, such that for any principal leaf $L$ and any $x\in \nu_pL$ the curvature operator $R$ on $M$ satisfies $\operatorname{tr}|_{T_pL} R(\cdot,x)x > 0$. Given a relatively compact open neighbourhood $P\subseteq L$ of $p$, define:
\begin{enumerate}
\item $U_P\subset M$ the open set $U_P=\bigcup_{q\in P} C_q$, where $C_q$ is the open Weyl chamber through $q$.
\item $(U_P,\fol_{P})$ the foliation by the intersections $L'\cap U_P$, for $L'\in \fol$. 
\end{enumerate}
It is easy to see that the inclusion $U_P\to M$ induces an injection $U_P/\fol_P\to M/\fol$ which is a homeomorphism onto the interior of $M/\fol$, and $\fol_P$ is \emph{full} in the sense that for any leaf $L'\cap U_P$ in $\fol_P$ there exists some $\epsilon>0$ such that the normal exponential map of $L'\cap U_P$ is well defined in $U_P$ up to distance $\epsilon$. In particular, for any leaf $L'\cap U_P$ of $\fol_P$, and any parallel normal vector field $X$ along $L'\cap U_P$, The map $\phi_{tX}:L'\cap U_P\to U_P$ is well defined for all $t$ in some interval around $0$, and in this case it is a diffeomorphism onto a leaf of $\fol_P$. This was the crucial property used in Proposition \ref{P:volume-concavity}, and allows to prove that the function
\[
V_P:U_P/\fol_P\to \RR\qquad V_P(p_*)=\operatorname{vol}\big(\pi^{-1}(p_*)\cap U_P\big)^{1\over n}
\]
is strictly concave on $U_P/\fol_P$ and approaches zero {}{towards the boundary of its closure in $M/\fol$}. In particular,{}{if the closure of $U_P/\fol_P$, i.e. $M/\fol$, is compact}, there is a unique leaf achieving the maximum volume.
\end{remark}

\begin{proof}[Proof of Theorem \ref{T:MCF}]
Let $(M,\fol)$ be an isoparametric foliation as in Theorem \ref{T:MCF}, and notice that the condition $Ric_M(x)>Ric_\Sigma(x)$ for all $x\in T_pM$ tangent to $\Sigma$, is equivalent to $\operatorname{tr}|_{\V_p}R>0$.

Since by Theorem \ref{T:parallel-H} the mean curvature vector of the regular leaves of $(M,\fol)$ is parallel, it projects to a vector field $H_*$ on the regular part of $M/\fol$. Furthermore the mean curvature flow $f(t,\cdot)$ starting from a regular leaf $L=\pi^{-1}(p_*)$ of $\fol$ flows through regular leaves of $\fol$, and in fact $L_t:=f(t,L)=\pi^{-1}(\gamma_*(t))$ where $\gamma_*$ is the integral curve of $H_*$ with $\gamma_*(0)=p_*$.

If the leaves of $\fol$ are not compact, consider $(U_P,\fol|_{U_P})$ as in Remark \ref{R:noncompact-case}. We have the following properties:
\begin{enumerate}
\item Given a leaf $L'\in \fol$ and $L'\cap U_P$ the corresponding leaf in $\fol_P$, clearly the mean curvature vector field of $L'\cap U_P$ equals the mean curvature vector field of $L'$, restricted to $L'\cap U_P$. In particular, the mean curvature vector field of $(U_P,\fol|_{U_P})$ is still basic.
\item Since $U_P$ is a union of (open sets of) sections, and the mean curvature vector field $H$ of $\fol_P$ is everywhere horizontal, the flow of $H$ moves leaves of $\fol_P$ onto leaves of $\fol_P$.
\item Since the mean curvature vector field of any regular leaf $L'$ is basic, then $L'$ is minimal if and only if $L'\cap U_P$ is minimal.
\end{enumerate}
By the properties above, $f_{U_P}(t,L\cap U_P)=L_t\cap U_P$ is the solution for the mean curvature flow starting from $L\cap U_P\in \fol|_{U_P}$ if and only if  $f(t,L)=L_t$ is the solution for the mean curvature flow starting at $L\in \fol$. In particular, up to replacing $(M,\fol)$ with $(U_P,\fol|_{U_P})$, we can assume that the leaves have finite volume.
\\

We analyze the integral curves $c_*(t,\cdot)=\pi(f(t,\cdot))$ of the vector field $H_*$ on the manifold part of $M/\fol$ for $t<0$. In particular, studying the behaviour of $f(t,\cdot)$ as $t\to -\infty$ reduces to studying the integral curves of $-H_*$ for positive times. By Proposition 3.3 of \cite{AlexandrinoRadeschi-II}, as $t\to -\infty$ the flow $f(t,\cdot)$ escapes small tubular neighbourhoods of any singular leaf. Thus, there is a tubular neighborhood $U$ of the singular set of $M/\fol$ such that the integral curves of $-H_*$ starting from $(M/\fol)\setminus U$ stay in $(M/\fol)\setminus U$ for all time $t>0$. Since $M/\fol$ was assumed to be compact, $(M/\fol)\setminus U$ is contained in a compact set. Furthermore, the function $V:M/\fol\to \RR$ from Proposition \ref{P:volume-concavity} (resp. $V_P:U_P/\fol|_{U_P}\to \RR$) is a Lyapunov function for the flow of $-H_*$. In particular, the flow has a unique global attractor, that is the projection of the unique minimal regular leaf of $\fol$.
\end{proof}

\begin{remark}\label{R:conc-holds}
The conditions in Proposition  \ref{P:volume-concavity} are easily seen to be satisfied in the following situations:
\begin{itemize}
\item $M$ is compact with $Ric_M>0$ and $(M,\fol)$ is hyperpolar.
\item $M$ is compact with $sec_M>0$ and $(M,\fol)$ is polar.
\end{itemize}
Furthermore, Proposition \ref{P:type3cpt} and Corollary \ref{L:cond-satisf} show that the condition above is satisfied for factors of type 3.
\end{remark}

Recall that, by Theorems 1.18 and 1.20 of \cite{Terng-II}, the leaves of an isoparametric foliation $(\RR^k,\fol)$ without trivial factors must be compact, and contained in concentric spheres. Furthermore restriction of $\fol$ to each sphere $S$ is still isoparametric, and by Theorem 1.1(2) of \cite{TerngLiu-II}  there is a unique regular leaf that is minimal in $S$.
The following proposition is a generalization of this result.

\begin{proposition}[Minimal leaves of polar foliations]\label{P:minl-leaves}
Let $(M,\fol)$ be a polar foliation on a simply connected symmetric space with non-negative curvature, and let $(M,\fol)= (M_{-1},\fol_{-1})\times (M_0,\fol_0) \times \prod_i (M_i,\fol_i)$ its decomposition into factors. Then:
\begin{enumerate}
\item All leaves of $\fol_{-1}$ are minimal.
\item $(M_0,\fol_0)$ has either one or no minimal regular leaves, depending on whether $M_0/\fol_0$  is compact or not.
\item Each of $(M_i,\fol_i)$ has exactly one minimal regular leaf.
\end{enumerate}
\end{proposition}
\begin{proof}

The first point is obvious. By Proposition \ref{P:real-splitting}, $(M_0,\fol_0)$ splits as a product of hyperpolar foliations $(M_1,\fol_1)\times (\RR^k,\fol_0^{nc})$, where for every $\fol_1$-horizontal direction $x$ in $(M_1,\fol_1)$, $\operatorname{tr}|_{\V}R>0$. In this case, $M_0/\fol_0$ is compact if and only if $k=0$.

If $k=0$ then $M_0=M_1$ has a minimal leaf by Theorem {\ref{T:parallel-H}} and Remark \ref{R:noncompact-case}. In fact, by Theorem {\ref{T:parallel-H}} a leaf $L$ is minimal if and only if for any relatively compact set $P$ of $L$ one had $H|_{L}$, which is equivalent to $P$ being a critical point for the volume functional on $U_P$ (as defined in Remark {\ref{R:noncompact-case}}). By Remark {\ref{R:noncompact-case}}, there exists exactly one such point.

If $k>0$, then it is well known that the leaves of $(\RR^k,\fol_0^{nc})$ are not minimal, and so neither are the leaves of $M_0$.

Finally, point 3) follows from Proposition \ref{P:volume-concavity} and Remark \ref{R:conc-holds} since factors of type 3 have positive Ricci curvature.
\end{proof}

\begin{proof}[Proof of Theorem \ref{T:ancient}]
Given $(M,\fol)$ polar foliation with compact quotient on a simply connected symmetric space with non-negative sectional curvature, it follows by Theorem \ref{T:parallel-H} that $\fol$ is isoparametric. Furthermore, by comparing Proposition \ref{P:real-splitting} and Corollary \ref{C:splitting}, it follows that $\fol$ satisfies $\operatorname{tr}|_{\V}R>0$ at regular points. Thus Theorem~\ref{T:MCF} applies, and the first part of the theorem is proved. The second part about uniqueness of the minimal leaf follows directly from Proposition \ref{P:minl-leaves}
\end{proof}

\end{document}